\documentclass[1p,preprint]{elsarticle}
\usepackage{amsmath}
\usepackage{amsthm}
\usepackage{amsfonts}
\usepackage{amssymb}
\usepackage{graphicx}
\usepackage{tikz}
\usepackage{url}
\usepackage{listings}
\usepackage{hhline}
\usepackage{xcolor}
\usepackage{booktabs}

\newtheorem{theorem}{Theorem}

\journal{Journal of Combinatorial Theory, Series A}

\begin{document}

\begin{frontmatter}

\title{Constraint Satisfaction Programming for the No-three-in-line Problem}

\author{Thomas Prellberg}
\ead{t.prellberg@qmul.ac.uk}
\address{School of Mathematical Sciences\\Queen Mary University of London\\United Kingdom}

\begin{abstract}
Using a constraint satisfaction approach, we exhibit configurations of $2n$ points on the $n\times n$
grid for all $n\le60$ with no three collinear. Consequently, the smallest $n$ for which it is unknown
whether $D(n)=2n$ increases from $47$ to $61$.
\end{abstract}


\end{frontmatter}

\section{Introduction}

The no-three-in-line problem, first posed more than one hundred years ago for a chessboard by Henry
Dudeney \cite[Puzzle 317]{dudeney1917}, asks for the maximum number $D(n)$ of lattice points that can
be selected from the $n\times n$ points of the square lattice so that no three lie on a common line
\cite{guy1981,moser1986}. An obvious upper bound is $D(n)\le2n$, and for large $n$ it is conjectured
that $D(n)<cn$ with $c=\pi/\sqrt3\approx1.81$ using a probabilistic estimate\cite{guy1968}. 
Best known constructions for large $n$ give 
configurations with $\frac 32n-o(n)$ lattice points based on modular hyperbolas \cite{hall1975}. Recent extensions include
generalizations to the infinite square grid \cite{nagy2023} and to higher dimensions \cite{lin2021}.
It is worth noting that the classical no-three-in-line problem discussed here is a special case of the general position problem
in combinatorial geometry and graphs, as surveyed in \cite{chandran2026}.

Until recently, examples with $D(n)=2n$ were known only for $n\le46$ and for $n=48,50,52$
\cite{flammenkamp1992,flammenkamp1998}. Harborth, Oertel, and Prellberg \cite{harborth1988} showed that
symmetry reduction is an effective way to reduce the search space; this was also key in
\cite{flammenkamp1992,flammenkamp1998}, where known solutions for large $n$ had $90^\circ$ rotational
symmetry for $n$ even, or $90^\circ$ rotational symmetry except for the diagonal for $n$ odd
(augmented by two symmetric points on one diagonal). The work presented here also makes use of these symmetries.

\begin{figure}[t!]
\centering
\includegraphics[width=.8\textwidth]{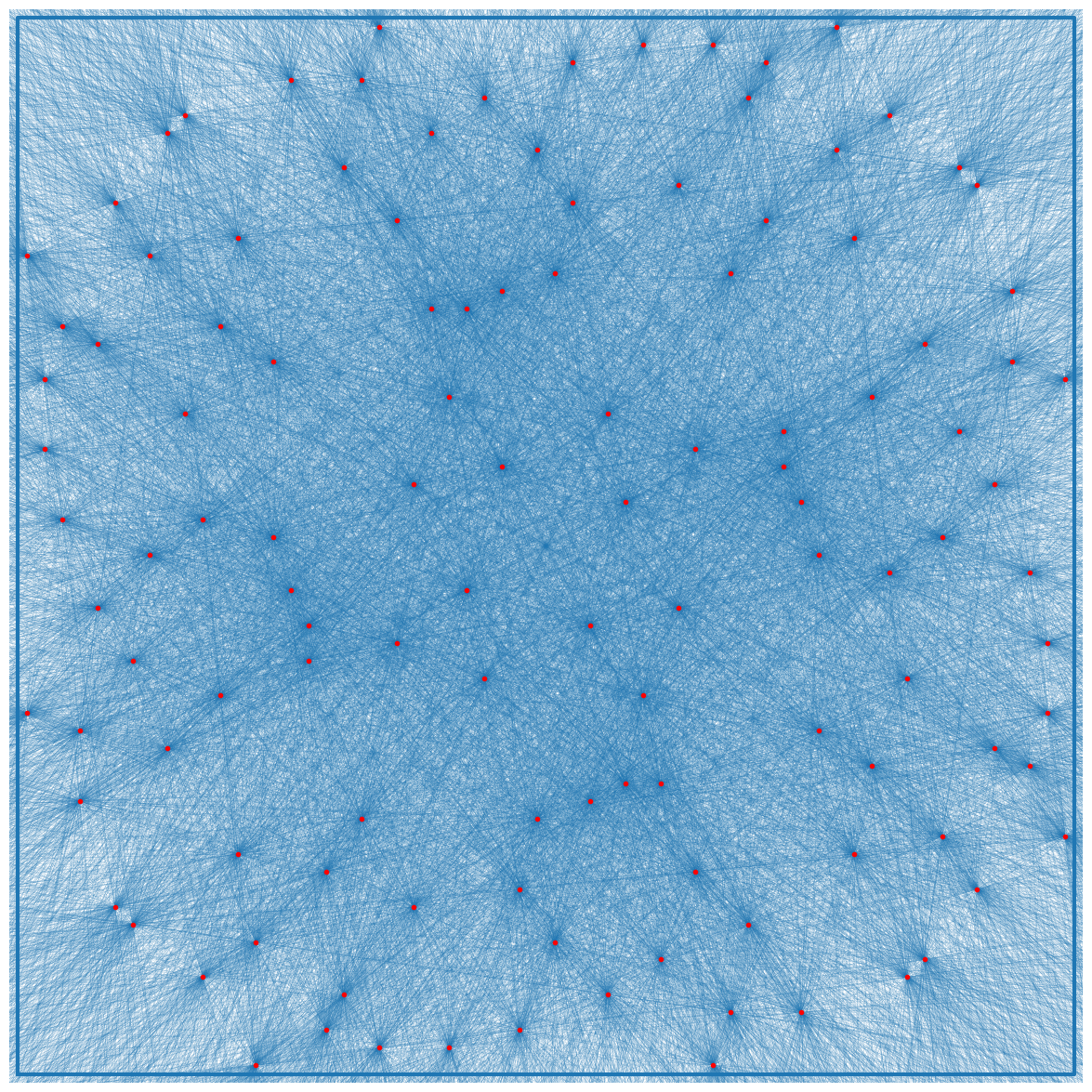}
\caption{$2n$ lattice points with no three in line for $n=60$.
The $120$ occupied sites (red) are shown together with the $\binom{120}{2}=7140$ straight-line
segments joining each pair of occupied sites (blue), illustrating the $90^\circ$ rotational symmetry.}
\label{fig60}
\end{figure}

There has been renewed interest in the computational aspect of this and related problems, applying
modern techniques such as integer programming, reinforcement learning, and AI
\cite{eppstein2018,charton2024,ramanathan2025}. Using a computer search based on constraint
satisfaction and CP-SAT \cite{cpsatlp}, we find configurations with $D(n)=2n$ for all $n\le60$.
Figure~\ref{fig60} shows the largest such configuration at $n=60$. Thus the smallest $n$ for which it
is unknown whether $D(n)=2n$ is $61$.

We conclude this section with detailing the probabilistic counting heuristic from \cite{guy1968}.
Let $G_n$ be the square grid of $n^2$ points, and let $t_n$ denote
the number of collinear triples in $G_n$,
\[
t_n=2\sum_{a=1}^{n-1}\sum_{b=1}^{n-1}(n-a)(n-b)\gcd(a,b)-\frac16n^2(n^2-1)=\frac{3}{\pi^2}n^4\log n+O(n^4).
\]
A uniformly random triple of points is collinear with probability $q_n=t_n/\binom{n^2}{3}$.
Now choose a $k$-subset $T\subset G_n$ uniformly at random. There are $\binom{k}{3}$
triples inside $T$. Approximating the events that these triples are collinear as
independent (a crude but convenient heuristic), we estimate
\[
\mathbb{P}(\text{$T$ contains no collinear triple})
\approx (1-q_n)^{\binom{k}{3}}.
\]
Multiplying by the number of candidate sets gives the heuristic count
\[
C(n,k)=\binom{n^2}{k}\left(1-q_n\right)^{\binom{k}{3}}.
\]
Next fix $\lambda>0$ and set $k=\lfloor\lambda n\rfloor$. Using standard asymptotics for
$\binom{n^2}{k}$ with $k=O(n)$ and the expansion $\log(1-q_n)=-q_n+o(q_n)$, one finds
\[
\log C(n,\lfloor\lambda n\rfloor)
\sim\left(\lambda-\frac{3\lambda^3}{\pi^2}\right)n\log n
=\frac{\lambda(\pi^2-3\lambda^2)}{\pi^2}\,n\log n.
\]
In particular, the leading coefficient changes sign at
\[
\lambda_c=\frac{\pi}{\sqrt3}\approx 1.81,
\]
suggesting that $2n$-point no-three-in-line configurations should become rare as $n$ grows. Moreover, using the exact expression for $t_n$, we find that the heuristic count $C(n,2n)$ becomes smaller than one once $n\geq493$. While this crude heuristic should not be taken too seriously, we note that $n=60$ is still far below the suspected threshold where $D(n)=2n$ might fail. 

Maybe more importantly, if one additionally incorporates rotational symmetry into this heuristic (180-degree symmetry suffices), the expected number of solutions drops immediately below one, even though solutions having rotational symmetry clearly exist. This exposes a clear weakness of this heuristic and casts doubts as to its validity.

\section{Results and formulations}

\subsection{Satisfiability formulation}

Let $[n]_0:=\{0,1,\dots,n-1\}\subset\mathbb Z$ and $G_n:=[n]_0^2\subset\mathbb Z^2$. Associate a binary
variable $x_{ij}\in\{0,1\}$ to each $(i,j)\in G_n$, where $x_{ij}=1$ indicates that $(i,j)$ is
occupied. For $q\ge2$ define
\[
\mathcal L^n_{\ge q}:=\{\ell\subset\mathbb R^2:\ell\text{ is an affine line and }|\ell\cap G_n|\ge q\}.
\]
The no-three-in-line problem can be written as the integer program
\[
D(n):=\max\left\{\sum_{(i,j)\in G_n}x_{ij}:\sum_{(i,j)\in\ell\cap G_n}x_{ij}\le2\ \ \forall \ell\in\mathcal L^n_{\ge3}\right\}.
\]
Since no-three-in-line implies $\leq2$ per row, any configuration with $2n$ points must have exactly two points in every row. Therefore we also consider the satisfiability variant obtained by enforcing this explicitly:
\begin{equation}
\label{CPSAT}
\mathcal F_n :=
\left\{ (x_{ij})_{(i,j)\in G_n}\in\{0,1\}^{G_n} \,:\,
\begin{aligned}
&\sum_{(i,j)\in \ell\cap G_n} x_{ij}\le 2
&&\forall \ell\in \mathcal L^n_{\ge 3},\\
&\sum_{i\in[n]_0} x_{ij}=2
&&\forall j\in[n]_0
\end{aligned}
\right\}.
\end{equation}

\subsection{Existence up to $n=60$}

\begin{theorem}
For $2\le n\le60$ there exists a configuration in $\mathcal F_n$. In particular, $D(n)=2n$ for
$2\le n\le60$.
\end{theorem}

\begin{proof}
For $n\le46$ and for $n\in\{48,50,52\}$ the existence of configurations with $D(n)=2n$ is given in
\cite{flammenkamp1992,flammenkamp1998}. Table~\ref{table1} supplies configurations in $\mathcal F_n$
for the remaining values $n\in\{47,49,51,53,54,55,56,57,58,59,60\}$. Since $D(n)\le2n$, this implies
$D(n)=2n$ for all $2\le n\le60$.
\end{proof}

\begin{table}[!ht]
\centering
\begin{tabular}{|p{0.03\linewidth}||p{0.95\linewidth}|}
\hline
$n$&\\
\hhline{|=||=|}
47&(0, 17), (0, 21), (1, 12), (3, 8), (4, 19), (5, 9), (5, 21), (6, 22), (7, 4), (7, 16), {\bf(8, 8)}, (10, 2), (10, 12), (11, 3), (11, 15), (13, 2), (14, 19), (15, 9), (16, 18), (17, 14), (18, 6), (20, 1), (20, 13), (23, 22)\\
\hline
49&(0, 13), (1, 18), (2, 9), (3, 5), (4, 2), (5, 20), (6, 16), (6, 18), (10, 4), (10, 7), (11, 15), (11, 19), (12, 9), (13, 14), {\bf(15, 15)}, (16, 7), (17, 0), (17, 21), (19, 14), (20, 8), (21, 22), (22, 3), (23, 1), (23, 12), (24, 8)\\
\hline
51&{\bf(0, 0)}, (0, 20), (2, 13), (4, 14), (4, 23), (5, 9), (6, 5), (7, 18), (10, 9), (10, 12), (11, 8), (13, 22), (14, 6), (15, 1), (15, 22), (16, 2), (16, 8), (17, 12), (18, 23), (19, 3), (20, 17), (21, 3), (21, 11), (24, 1), (24, 7), (25, 19)\\
\hline
53&(0, 16), (0, 20), (2, 20), (2, 25), (3, 24), {\bf(5, 5)}, (5, 8), (6, 14), (6, 24), (7, 8), (7, 21), (9, 10), (11, 3), (13, 9), (13, 19), (14, 17), (15, 1), (15, 22), (17, 12), (18, 1), (18, 12), (19, 4), (21, 4), (22, 11), (23, 10), (23, 25), (26, 16)\\
\hline
54&(0, 22), (0, 24), (2, 16), (3, 7), (3, 24), (4, 15), (5, 11), (6, 17), (8, 25), (9, 5), (9, 6), (11, 14), (12, 2), (12, 13), (13, 21), (17, 4), (18, 10), (18, 16), (19, 1), (19, 10), (20, 8), (20, 15), (22, 1), (23, 7), (23, 14), (25, 26), (26, 21)\\
\hline
55&(2, 14), (3, 18), {\bf(4, 4)}, (4, 18), (5, 9), (6, 2), (6, 11), (7, 26), (8, 19), (8, 20), (9, 25), (10, 13), (11, 13), (12, 3), (12, 21), (14, 16), (15, 10), (17, 20), (17, 24), (19, 1), (22, 5), (22, 7), (23, 1), (23, 15), (24, 16), (25, 0), (26, 0), (27, 21)\\
\hline
56&(1, 5), (2, 13), (3, 1), (3, 27), (5, 26), (6, 17), (7, 23), (8, 20), (9, 10), (9, 14), (10, 19), (12, 0), (12, 21), (14, 11), (16, 6), (16, 21), (17, 24), (18, 4), (18, 11), (19, 15), (20, 7), (22, 2), (22, 4), (23, 24), (25, 8), (25, 15), (26, 13), (27, 0)\\
\hline
57&(1, 23), (3, 24), (4, 6), (4, 24), (5, 16), (5, 18), (6, 3), (7, 11), (8, 16), (8, 26), (9, 2), {\bf(10, 10)}, (12, 0), (12, 21), (13, 7), (13, 17), (14, 15), (14, 26), (15, 17), (19, 10), (19, 11), (20, 1), (20, 25), (22, 21), (23, 2), (25, 0), (27, 9), (27, 22), (28, 18)\\
\hline
58&(1, 14), (1, 19), (2, 19), (3, 26), (3, 27), (5, 15), (6, 5), (7, 18), (8, 4), (10, 7), (11, 4), (12, 15), (13, 11), (13, 18), (14, 22), (16, 21), (17, 2), (17, 8), (20, 6), (20, 24), (21, 0), (22, 16), (23, 23), (25, 9), (25, 10), (26, 24), (27, 0), (28, 9), (28, 12)\\
\hline
59&(0, 12), (1, 26), (2, 27), (3, 15), (4, 7), (5, 21), (5, 25), (7, 14), (8, 16), (8, 26), {\bf(9, 9)}, (9, 11), (10, 4), (10, 27), (12, 16), (13, 6), (14, 18), (15, 25), (18, 17), (19, 6), (20, 1), (21, 11), (22, 17), (23, 3), (23, 20), (24, 0), (24, 13), (28, 19), (28, 22), (29, 2)\\
\hline
60&(0, 20), (3, 15), (3, 19), (4, 26), (5, 9), (6, 8), (6, 23), (7, 29), (8, 18), (10, 5), (11, 21), (12, 12), (13, 0), (13, 7), (15, 27), (16, 23), (16, 25), (17, 2), (17, 11), (18, 4), (19, 14), (20, 1), (21, 24), (22, 9), (24, 1), (25, 27), (26, 22), (28, 2), (28, 10), (29, 14)\\
\hline
\end{tabular}
\caption{Representatives of occupied-site orbits. Here $\rho(i,j)=(j,n-1-i)$ is the $90^\circ$ rotation about the
center $(\frac{n-1}2,\frac{n-1}2)$. For $n$ odd, bold entries lie on the main diagonal and generate $2$-cycles under
$\rho^2$; the anti-diagonal is empty. The full configuration is obtained by applying $\rho$ (off-diagonal entries) or
$\rho^2$ (bold diagonal entries for odd $n$) to each representative.}
\label{table1}
\end{table}

\section{Symmetry reduction}

The configurations in Table~\ref{table1} have additional rotational symmetry. Let $\rho:G_n\to G_n$
denote the $90^\circ$ rotation about the center of the grid,
\[
\rho(i,j)=(j,n-1-i),
\]
and let $\langle\rho\rangle$ be the cyclic group it generates. Table~\ref{table1} lists occupied sites
only in a fundamental domain $H_n\subseteq G_n$; the full configuration is obtained by taking the
appropriate rotational orbit of each listed site (full $\langle\rho\rangle$-orbits off-diagonal and
$\langle\rho^2\rangle$-orbits on the main diagonal when $n$ is odd).

Concretely, define
\[
H_n:=\begin{cases}
[\,n/2\,]_0\times[\,n/2\,]_0,&n\text{ even},\\[2mm]
[\, (n+1)/2 \,]_0\times[\, (n-1)/2 \,]_0,&n\text{ odd}.
\end{cases}
\]
Under $\rho$, a generic site has a $\langle\rho\rangle$-orbit of size $4$. When $n$ is odd we require
the anti-diagonal $\{(i,n-1-i)\}$ to be empty, and diagonal representatives are expanded only under
$\langle\rho^2\rangle$; these are indicated in boldface in Table~\ref{table1}.

For even $n$, every site in $H_n$ has a $\langle\rho\rangle$-orbit of size $4$, so choosing $n/2$
orbit representatives in $H_n$ yields $4\cdot(n/2)=2n$ occupied sites in $G_n$. For odd $n$,
off-diagonal sites in $H_n$ have $\langle\rho\rangle$-orbits of size $4$, whereas a diagonal site
$(i,i)$ has a $\langle\rho^2\rangle$-orbit of size $2$. Thus choosing $(n-1)/2$ off-diagonal
representatives and one diagonal representative yields $4\cdot((n-1)/2)+2=2n$ occupied sites in
$G_n$.

To encode these symmetry requirements in the satisfiability model, introduce binary variables
$y_{ij}\in\{0,1\}$ for $(i,j)\in H_n$ and link them to occupancy variables $x_{kl}\in\{0,1\}$ on $G_n$
as follows. If $n$ is even, then every $\langle\rho\rangle$-orbit meets $H_n$ in exactly one site, and
we set
\[
x_{\rho^t(i,j)}:=y_{ij}\qquad\text{for all }(i,j)\in H_n\text{ and }t\in\{0,1,2,3\}.
\]
If $n$ is odd, we keep the anti-diagonal empty and distinguish off-diagonal and diagonal sites. For
$(i,j)\in H_n$ with $i\ne j$ set
\[
x_{\rho^t(i,j)}:=y_{ij}\qquad\text{for }t\in\{0,1,2,3\},
\]
while for diagonal sites $(i,i)\in H_n$ set
\[
x_{i,i}=x_{\rho^2(i,i)}=x_{n-1-i,n-1-i}:=y_{ii},
\]
and enforce
\[
x_{i,n-1-i}=0\qquad\text{for all }i\in[n]_0.
\]

Equivalently, for $(i,j)\in H_n$ define
\[
O(i,j):=\begin{cases}
\{\rho^t(i,j):t=0,1,2,3\},&\text{if $n$ is even or $i\ne j$},\\[1mm]
\{(i,i),\rho^2(i,i)\}=\{(i,i),(n-1-i,n-1-i)\},&\text{if $n$ is odd and $i=j$},
\end{cases}
\]
and for odd $n$ let $A_n:=\{(i,n-1-i):i\in[n]_0\}$ denote the anti-diagonal. Then all sites in
$O(i,j)$ share the common value $y_{ij}$, and all sites in $A_n$ are fixed to $0$.

Let $\Gamma_n:=\langle\rho\rangle$ if $n$ is even and $\Gamma_n:=\langle\rho^2\rangle$ if $n$ is odd.
Let $\mathcal L^{n,\Gamma}_{\ge3}\subseteq \mathcal L^n_{\ge3}$ be a choice of one representative from
each $\Gamma_n$-orbit of lines in $\mathcal L^n_{\ge3}$.

For any affine line $\ell\subset\mathbb R^2$ define
\[
c_{\ell}(i,j):=\bigl|\ell\cap O(i,j)\bigr|\qquad((i,j)\in H_n),
\]
and for any row index $r\in[n]_0$ define
\[
d_{r}(i,j):=\bigl|\{(k,r)\in O(i,j)\}\bigr|\qquad((i,j)\in H_n).
\]
The reduced satisfiability model becomes
\begin{equation}
\label{CPSATred}
\mathcal F_n^{\mathrm{sym}} :=
\left\{ (y_{ij})_{(i,j)\in H_n}\in\{0,1\}^{H_n} \,:\,
\begin{aligned}
&\sum_{(i,j)\in H_n} c_{\ell}(i,j)\,y_{ij}\le 2
&&\forall \ell\in \mathcal L^{n,\Gamma}_{\ge 3},\\
&\sum_{(i,j)\in H_n} d_{r}(i,j)\,y_{ij}=2
&&\forall r\in[n]_0
\end{aligned}
\right\}.
\end{equation}

\section{Implementation and computational setup}

Although the no-three-in-line problem is classical, obtaining new exact results via a general-purpose CP-SAT solver is not a black-box exercise. The formulation given earlier is necessary but, in practice, far from sufficient: whether an instance is solved quickly or not at all can hinge on details of solver behaviour that are largely invisible at the level of the mathematical model, such as the interaction between propagation, cutting, branching, restarts, and the solver’s internal portfolio of search strategies. These mechanisms are not designed with this particular combinatorial geometry problem in mind, and their effective use therefore requires problem-specific experimentation and an understanding of what aspects of the instance structure the solver is able (or unable) to exploit.

For this reason we describe our implementation choices and computational protocol in some detail. Beyond enabling verification and reproducibility, these details are intended to be methodological: many combinatorial problems admit compact constraint formulations yet remain computationally intractable unless one knows how to present the structure to a modern solver and how to interpret its stochastic and heavy-tailed runtime behaviour. We expect that the modelling decisions, symmetry handling, and restart/parallelisation strategy discussed below will be relevant to researchers aiming to apply CP-SAT-type technology to other exact search problems in combinatorics, where the same “model correctness versus solver effectiveness” gap frequently determines what can be proved computationally.

\subsection{Hardware and software}

Most computations reported here were carried out on an AMD EPYC 9965 CPU (192 cores, 384 threads) with
2.2\,TB RAM. Unless stated otherwise, run times refer to this architecture and are measured as
wall-clock time to first solution when running $384$ independent solver instances in parallel.
All CP-SAT runs use a single search worker per instance (i.e.\ no internal parallelism within a CP-SAT run).
In cases when runs were terminated after the maximal job length of ten days, or when runs were distributed
over cores of several other HPC machines, a heuristic composite runtime equivalent was used.

\subsection{Constraint generation and deduplication}

Implementing the satisfiability model~\eqref{CPSAT} is straightforward, and a closely related mixed
integer programming formulation was previously solved with Gurobi \cite{eppstein2018}. The CP-SAT
implementation is similar: introduce binary variables $x_{ij}$ and add the linear constraints
$\sum_{(i,j)\in\ell\cap G_n}x_{ij}\le2$ for each $\ell\in\mathcal L^n_{\ge3}$. The resulting model has
$n^2$ binary variables and $|\mathcal L^n_{\ge3}|$ line constraints. Moreover
$|\mathcal L^n_{\ge3}|=\Theta(n^4)$; more precisely,
$|\mathcal L^n_{\ge3}|\sim\frac{5}{12\pi^2}n^4$ as $n\to\infty$ \cite{haukkanen2012}. For comparison,
$\binom{n^2}{2}$ is a crude upper bound on the number of distinct lines determined by unordered pairs
of grid points, so $|\mathcal L^n_{\ge3}|/\binom{n^2}{2}\to5/(6\pi^2)\approx0.084$.

In the implementation, each line constraint is represented by its finite intersection set
$\ell\cap G_n$. Concretely, we build a dictionary whose keys are canonical encodings of distinct
intersection sets $\ell\cap G_n$, and whose values store the corresponding lists of lattice points.
The resulting collection of constraints is then passed to CP-SAT as linear inequalities. The
computational cost of generating these constraints and performing this deduplication is negligible
compared with the subsequent solver runtime for the instances considered here.

After rewriting constraints in terms of orbit variables $y_{ij}$ in the symmetry-reduced model
\eqref{CPSATred}, further accidental coincidences can occur: distinct representatives
$\ell\in\mathcal L^{n,\Gamma}_{\ge3}$ may induce the same incidence vector
$(c_{\ell}(i,j))_{(i,j)\in H_n}$ on $H_n$ and hence the same linear inequality. We therefore optionally
deduplicate again at this stage by keeping one representative per distinct incidence vector (CP-SAT
tolerates the redundancy either way). In addition, some line inequalities become tautological after
symmetry reduction (for example, if enforcing an empty anti-diagonal reduces $\ell\cap G_n$ to at
most two unfixed orbit sites), and these can be discarded outright since
$\sum_{(i,j)\in H_n} c_{\ell}(i,j)\,y_{ij}\le2$ then holds automatically for $y_{ij}\in\{0,1\}$.

\subsection{Model sizes}

\begin{table}[!ht]
\centering
\begin{tabular}{|r|rr|rr|}
\hline
 & \multicolumn{2}{c|}{Binary variables} & \multicolumn{2}{c|}{Constraints} \\
\hline
$n$ & no restriction & symmetry & no restriction & symmetry \\
\hline
2  & 4    & 1   & 2      & 2 \\
3  & 9    & 2   & 8      & 3 \\
4  & 16   & 4   & 14     & 6 \\
5  & 25   & 6   & 32     & 12 \\
$\vdots$ & $\vdots$ & $\vdots$ & $\vdots$ & $\vdots$ \\
56 & 3136 & 784 & 415230 & 103868 \\
57 & 3249 & 812 & 446296 & 118241 \\
58 & 3364 & 841 & 476358 & 119160 \\
59 & 3481 & 870 & 510756 & 135183 \\
60 & 3600 & 900 & 546354 & 136660 \\
\hline
\end{tabular}
\caption{Model sizes for the satisfiability formulation~\eqref{CPSAT} (``no restriction'') and its rotationally
symmetric reduction~\eqref{CPSATred} (``symmetry''). The table reports the number of Boolean decision variables and the
number of linear constraints in each model. Constraint counts include the $n$ row equalities enforcing two occupied sites
per row; horizontal line inequalities are not counted separately since they are subsumed by these equalities. In the
symmetry-reduced model we additionally discard line inequalities that become tautological after rewriting in orbit
variables (e.g.\ when fewer than three unfixed orbit sites remain on a line).}
\label{model}
\end{table}

Table~\ref{model} illustrates the substantial reduction in model size obtained by enforcing rotational symmetry.
The number of variables drops from $n^2$ to $|H_n|$, i.e.\ by a factor close to four (exactly $4$ for even $n$, and
$(n^2-1)/4$ for odd $n$). The number of constraints is reduced by roughly the same factor, reflecting both the smaller
variable set and the restriction to a representative family of line constraints (with further deduplication after
rewriting in orbit variables).

\subsection{Parallel protocol and termination}

Our parallel protocol is ``run-until-first-success''. Specifically, we launch $M=384$ independent CP-SAT runs, each with
a fresh random seed and a single internal search worker. Runs proceed independently until the first solver reports a
feasible solution; at that point the experiment terminates and all remaining solver instances are stopped. Reported wall
times therefore correspond to $T_M=\min\{T^{(1)},\dots,T^{(M)}\}$, measured from a common launch time to the first success.

\section{Solve times and scaling}

\subsection{Baseline timings}

\begin{figure}[!ht]
\centering
\includegraphics[width=0.9\linewidth]{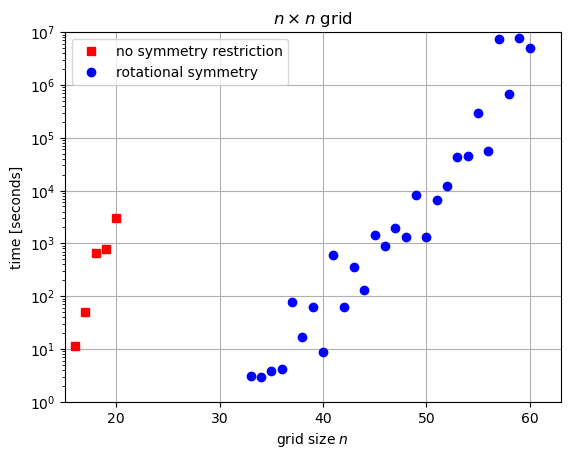}
\caption{Wall-clock time (log scale) to the first solution in a single experiment consisting of
$384$ independent CP-SAT runs executed in parallel on an AMD EPYC 9965 CPU. Red squares: direct
model~\eqref{CPSAT} (no symmetry restriction). Blue circles: symmetry-reduced
model~\eqref{CPSATred} enforcing rotational symmetry.} 
\label{times}
\end{figure}

As can be seen from Figure~\ref{times}, without symmetry restrictions the solver finds configurations
up to about $n=20$ within roughly one hour of wall-clock time in this experiment. Restricting the
search space to configurations with $90^\circ$ rotational symmetry (modified as described above for
odd $n$) reduces the model size substantially while still leaving solutions within the reduced model.
Together, these reductions make much larger instances tractable: as Figure~\ref{times} shows,
configurations for $n=50$ can be found within about one hour.

Since large runs are inherently stochastic, we now quantify the distribution of solve times for the
symmetry-reduced model~\eqref{CPSATred} and use it for a tentative extrapolation to larger $n$.

\subsection{Empirical solve-time distributions and parallelism}

Let $T$ denote the wall-clock time to the first solution in a single CP-SAT run (with a fresh random
seed), and write $F(t)=\mathbb P(T\le t)$ for its cumulative distribution function and
$S(t)=1-F(t)$ for the survival function. When running $M$ independent solver instances in parallel,
the time to the first solution is
\[
T_M:=\min\{T^{(1)},\dots,T^{(M)}\},
\]
and therefore
\begin{equation}
\label{FM}
S_M(t)=\mathbb P(T_M>t)=S(t)^M,
\qquad\text{equivalently}\qquad
F_M(t)=1-(1-F(t))^M.
\end{equation}
Thus, the small-$t$ behaviour of $F(t)$ controls the practically relevant distribution $F_M(t)$.

\paragraph{Data collection protocol.}
For $n=33,\dots,44$ we generated empirical solve-time distributions by first running a brief pilot to
identify a cutoff time $\tau$ such that $F_{384}(\tau)$ was roughly $0.99$. We then ran a larger batch
of independent single-run experiments with this single cutoff. Over the time windows shown in
Figures~\ref{fig:N34} and \ref{fig:N44}, censoring effects are negligible on the scales relevant for
$F_{384}(t)$, and we compute $F(t)$ directly from the observed completion times. (If heavy censoring
were present---for example, if a substantial fraction of runs timed out before completion---then a
Kaplan--Meier estimator would be the appropriate replacement; we did not need this here.)

\begin{figure}[!ht]
\centering
\includegraphics[width=0.9\linewidth]{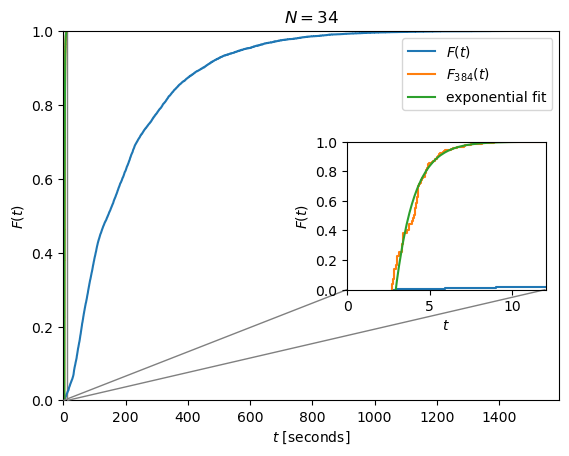}
\caption{Empirical cumulative distribution function $F(t)$ of the time to success for a single CP-SAT
run (blue), obtained from $10{,}436$ independent runs of~\eqref{CPSATred} at $n=34$.
Also shown is the derived CDF $F_{384}(t)$ for the first success among $M=384$ parallel runs (orange),
computed via~\eqref{FM}, and a shifted-exponential fit~\eqref{expmodel} to $F_{384}(t)$ near the
origin (green). The inset magnifies the small-time region relevant for $F_{384}$.}
\label{fig:N34}
\end{figure}

Figure~\ref{fig:N34} illustrates this transformation for $n=34$. The blue curve is the empirical
single-run CDF $F(t)$, while the orange curve is the derived $F_{384}(t)$ computed from~\eqref{FM}.
The latter concentrates heavily at small times: a small success probability per run is amplified into
a large success probability for the parallel experiment.

\subsection{A shifted-exponential model}

Motivated by the near-linearity of $-\log(1-F_{384}(t))$ over the early-time regime, we model
$T_M$ by a shifted exponential distribution,
\begin{equation}
\label{expmodel}
F^{\exp}_M(t)=
\begin{cases}
0, & t\le t_0,\\[2mm]
1-\exp\!\left(-\dfrac{M(t-t_0)}{t_1}\right), & t>t_0,
\end{cases}
\end{equation}
where $t_0$ captures an effective deterministic overhead (model construction, presolve, initial
propagation, etc.) and $t_1$ is a characteristic ``delay'' scale. If the underlying single-run time
$T$ itself were shifted exponential with the same $t_0$ and $t_1$, then~\eqref{expmodel} would follow
exactly from~\eqref{FM}; here we use~\eqref{expmodel} primarily as an accurate description of the
empirical $F_M(t)$ for small~$t$.

For $T_M$ distributed as~\eqref{expmodel}, the mean and quantiles are explicit. In particular,
\begin{equation}
\label{meanquant}
\mathbb E[T_M]=t_0+\frac{t_1}{M},
\qquad
t_p:=F_M^{-1}(p)=t_0+\frac{t_1}{M}\log\!\left(\frac{1}{1-p}\right).
\end{equation}

\begin{figure}[!ht]
\centering
\includegraphics[width=0.9\linewidth]{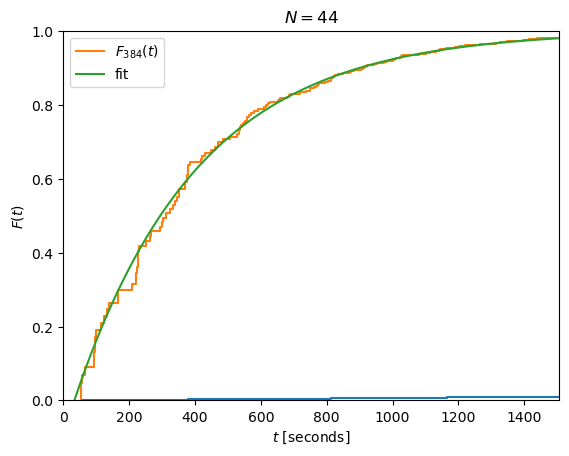}
\caption{For $n=44$: derived CDF $F_{384}(t)$ for the first success among $384$ parallel runs (orange)
together with a shifted-exponential fit~\eqref{expmodel} (green). On this time window the empirical
single-run CDF $F(t)$ is extremely small and is therefore not visually prominent.}
\label{fig:N44}
\end{figure}

\subsection{Parameter growth and extrapolation}

Table~\ref{table2} reports the fitted parameters $(t_0,t_1)$ for even $n$ from $34$ to $44$, together
with the implied mean time to first solution for $M=384$ parallel runs, $\mathbb E[T_{384}]$ from
\eqref{meanquant}. While the fitted threshold $t_0$ grows moderately, the delay scale $t_1$ grows
rapidly; for $n=44$ it already exceeds $40$ hours, even though $\mathbb E[T_{384}]$ remains on the
order of minutes.

\begin{table}[!ht]
\centering
\begin{tabular}{r|r|r|r}
$n$ & $t_0$ [s] & $t_1$ [s] & $t_0+t_1/384$ [s]\\
\hline
34 &  2.92 &    550.61 &   4.36\\
36 &  3.70 &   1721.67 &   8.18\\
38 &  6.77 &   6029.78 &  22.47\\
40 & 10.96 &  12016.51 &  42.25\\
42 & 29.79 &  24653.03 &  93.99\\
44 & 33.24 & 144613.52 & 409.84\\
\end{tabular}
\caption{Fitted threshold time $t_0$ and delay time scale $t_1$ for even $n=34,36,\dots,44$, using the
shifted-exponential model~\eqref{expmodel}. The final column is the implied mean time to first
solution when running $M=384$ solvers in parallel, cf.~\eqref{meanquant}.}
\label{table2}
\end{table}

For practical scheduling, quantiles are more relevant than the mean. Using~\eqref{meanquant}, the
median and $98\%$ quantile of $T_M$ are
\[
t_{0.5}=t_0+\frac{t_1}{M}\log 2,
\qquad
t_{0.98}=t_0+\frac{t_1}{M}\log 50.
\]
Figure~\ref{even} plots, on a semi-logarithmic scale, (i) observed single-run solve times for the
instances considered, and (ii) the estimated $50\%$ and $98\%$ completion times for $M=384$ parallel
runs. Least-squares fits of $\log t$ against $n$ over the measured range suggest an approximately
exponential growth of these quantiles with $n$. Extrapolating these fits gives a rough indication of
which sizes may be reachable within a fixed wall-clock budget (here indicated by a ten-day line).
We emphasize that this extrapolation is heuristic: modest curvature in the semi-log plot would
translate into large changes at the extrapolated sizes.

\begin{figure}[!ht]
\centering
\includegraphics[width=0.9\linewidth]{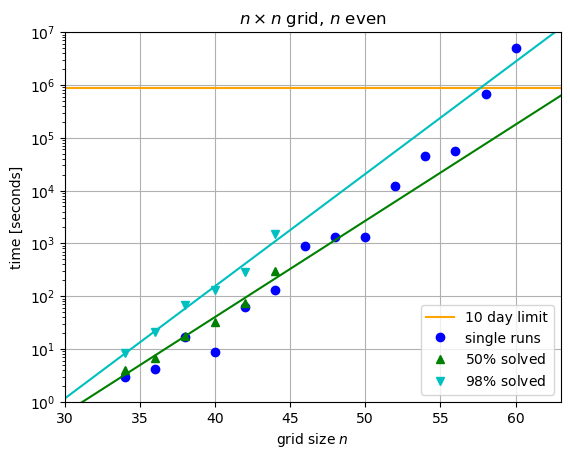}
\caption{Even sizes: semi-log plot of solve-time statistics versus grid size $n$. Blue circles:
observed wall-clock times to first solution in single runs. Green triangles and cyan inverted
triangles: estimated $50\%$ and $98\%$ completion times for $M=384$ parallel runs, derived from the
shifted-exponential fit~\eqref{expmodel} via~\eqref{meanquant}. Solid lines: least-squares fits of
$\log t$ versus $n$. The horizontal line indicates a ten-day wall-clock budget. Reported times over 10 days are aggregated from multiple job runs.}
\label{even}
\end{figure}

Even and odd $n$ behave differently under our enforced symmetry (for odd $n$ only
$\Gamma_n=\langle\rho^2\rangle$ is imposed on diagonal variables), so we analyse odd sizes separately.
Figure~\ref{odd} presents the analogous semi-log plot for odd $n$ (with fits obtained over the
measured range). The extrapolations are consistent with our computational results: odd sizes up to
$n=59$ are found within the ten-day horizon under the same parallel regime.

\begin{figure}[!ht]
\centering
\includegraphics[width=0.9\linewidth]{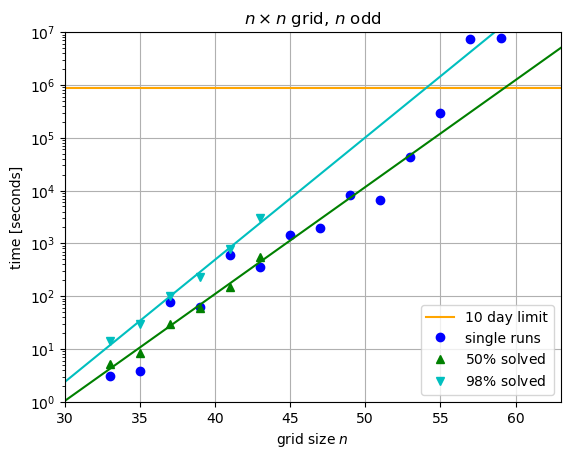}
\caption{Odd sizes: solve-time statistics versus grid size $n$, shown as in
Figure~\ref{even}. Because the enforced symmetry is weaker for odd $n$, the corresponding
solve-time growth differs, motivating separate fits.}
\label{odd}
\end{figure}

\subsection{Attempts at further speedups}

Given the effectiveness of CP-SAT on this problem, we explored several refinements. CP-SAT exposes a
large number of parameters, but in our experiments parameter changes did not produce improvements
beyond statistical variation. In line with standard practice, we also ran portfolios of parameter
choices (a ``hedging'' strategy) for large $n$, again without a consistent speedup.

We additionally tested: (i) mild guidance via objective functions or branching priorities derived
from empirical density profiles; (ii) for even $n$, constraints fixing the number of occupied sites
on the diagonals; (iii) ``anchoring'' selected sites; and (iv) warm starts based on partial
configurations. None of these improved performance reliably, and some were detrimental. In
particular, warm-starting an instance at size $n+2$ by embedding a solution at size $n$ often caused
the search to stall.

Overall, the best-performing strategy was to allow the solver to operate with minimal external
interference. While CP-SAT is partly a black box, its strength is widely attributed to the dynamic
generation of cutting planes and learned clauses during search. External restrictions can change the
search trajectory and may inhibit this mechanism, offering a plausible explanation for the negative
results of several ``helper'' constraints in our setting.

\section{Conclusion and outlook}

We have shown, using a symmetry-reduced CP-SAT formulation, that configurations with $2n$ points and
no three collinear exist on the $n\times n$ grid for all $n\le 60$. Combined with earlier results for
$n\le46$ and $n\in\{48,50,52\}$, this raises the smallest $n$ for which it is currently unknown
whether $D(n)=2n$ from $47$ to $61$.

From a computational perspective, enforcing rotational symmetry reduces the satisfiability model by
roughly a factor of four in both variables and constraints and makes sizes near $n=60$ tractable on
modern many-core hardware. Empirically, for parallel ``run-until-first-success'' experiments, the
early-time behaviour of the solve-time distribution is well described by a shifted-exponential model,
allowing quantiles relevant for scheduling to be estimated and (heuristically) extrapolated. The
resulting extrapolations are suggestive but should be interpreted cautiously: mild curvature in the
semi-log scaling plots would substantially change predictions at larger sizes.

The remaining open problem is to determine whether $D(61)=122$ (and beyond). We attempted $n=61$ and $n=62$ 
with the current method; however, no solution was found within $10^7$ seconds, so we are 
leaving this as an open case.

On the algorithmic side, an interesting direction is to identify additional structure (beyond rotational symmetry) that
reduces the search space without eliminating solutions, or to design hybrid approaches that exploit
learned constraints while preserving the solver's ability to generate effective clauses during
search.

\section*{Data and software availability}

The configurations listed in Table~\ref{table1} are representatives of rotational orbits and can be
expanded to full configurations as described in Section~3.

A reference implementation of the CP-SAT model described in this manuscript was generated from the
mathematical specification and implementation details in the text using an LLM-based coding assistant
(ChatGPT). The generated code was then run and validated against a verification program and against
the reported configurations and timings. This code and the configuration data from Table \ref{table1} 
are available at 
\begin{itemize}
    \item \url{https://github.com/ThomasPrellberg/no-three-in-line---CP-SAT}
\end{itemize} 
to facilitate reproducibility. Independent verification of each listed configuration takes seconds and does not rely on CP-SAT.

\section*{Declaration of competing interest}
The author declares no competing interests.

\section*{Acknowledgements}
This research utilized Queen Mary's Apocrita HPC facility, supported by QMUL Research-IT,
\url{http://doi.org/10.5281/zenodo.438045}.

\bibliographystyle{elsarticle-num}

\section*{\sloppy Declaration of generative AI and AI-assisted technologies in the manuscript preparation process}
During the preparation of this work the author used OpenAI ChatGPT (ChatGPT Plus) in order to generate Python code implementing (\ref{CPSAT}) and (\ref{CPSATred}), and to assist with editorial checks and content organization. After using this tool/service, the author reviewed and edited the content as needed and takes full responsibility for the content of the published article.

\end{document}